\documentclass[11pt,english,a4paper,reqno]{amsart}
\usepackage{amsmath,amssymb,amsfonts,amstext,amsthm}

\usepackage{babel}
\usepackage{latexsym}
\usepackage{ifthen}
\usepackage{xypic}

\usepackage[T1]{fontenc}
\usepackage{mathpazo}
\usepackage{hyperref}
\newtheorem{theorem}{Theorem}[section]
\newtheorem{lemma}[theorem]{Lemma}
\newtheorem{proposition}[theorem]{Proposition}
\newtheorem{remark}[theorem]{Remark}
\newtheorem{corollary}[theorem]{Corollary}
\newtheorem{claim}[theorem]{Claim}
\newtheorem{sublemma}[theorem]{Sub-Lemma}
\newtheorem{subcorollary}[theorem]{Sub-Corollary}
\theoremstyle{definition}
\newtheorem{example}[theorem]{Example}

\newcounter{subthm}[theorem]

\newcommand{\id}{{\rm id}}

\DeclareMathOperator{\coker}{coker}

\newcommand{\KC}{{\mathbb C}}

\newcommand\sE{{\mathcal E}}

\newcommand\sF{{\mathcal F}}

\newcommand\sH{{\mathcal H}}
\newcommand\sI{{\mathcal I}}

\newcommand\sL{{\mathcal L}}
\newcommand\sO{{\mathcal O}}

\newcommand\sS{{\mathcal S}}

\newcommand\sV{{\mathcal V}}
\newcommand\sX{{\mathcal X}}
\newcommand\sY{{\mathcal Y}}
\newcommand\sC{{\mathcal C}}

\newcommand\sA{{\mathcal A}}

\newcommand\bZ{{\mathbb Z}}
\newcommand\bC{{\mathbb C}}
\newcommand\bQ{{\mathbb Q}}

\newcommand\bF{{\mathbb F}}

\newcommand\bP{{\mathbb P}}

\begin{document}
\title{Contact K\"ahler Manifolds: Symmetries and Deformations}
%\author[K. Frantzen]{Kristina Frantzen}
%\address[Kristina Frantzen]{}
%\email{Kristina.Frantzen@uni-bayreuth.de}
\author[T. Peternell]{Thomas Peternell}
\address[Thomas Peternell]{Mathematisches Institut \\ Universit\"at Bayreuth \\ 95440 Bayreuth \\ Germany}
\email{thomas.peternell@uni-bayreuth.de}
\author[F. Schrack]{Florian Schrack}
\address[Florian Schrack]{Mathematisches Institut \\ Universit\"at Bayreuth \\ 95440 Bayreuth \\ Germany}
\email{florian.schrack@uni-bayreuth.de}
\begin{abstract}
We study complex compact K\"ahler manifolds $X$ carrying a
contact structure.
If $X$ is almost homogeneous and $b_2(X) \geq 2$, then $X$
is a projectivised tangent bundle (this was known in the projective
case even without assumption on the existence of vector fields).
We further show that a global projective deformation of the
projectivised tangent
bundle over a projective space is again of this type unless
it is the projectivisation of a special unstable bundle over a
projective space. Examples for these bundles are given in
any dimension.
\end{abstract}
\date{October 5, 2012}
\maketitle
\section{Introduction}

A contact structure on a complex manifold $X$ is in some sense the opposite of a foliation:
there is a vector bundle sequence
\[
0 \to F \to T_X \to L \to 0,
\]
where $T_X$ is the tangent bundle and $L$ a line bundle, with the additional property that the induced map
\[
\bigwedge^2F \to L, \ v \wedge w \mapsto [v,w]/F
\]
is everywhere non-degenerate. \\
Suppose now that $X$ is compact and K\"ahler or projective. 
If $b_2(X) = 1,$ then at least conjecturally the structure is well-understood: $X$ should arise as minimal orbit in the projectivised Lie algebra of contact automorphisms. 
Beauville \cite{Be98} proved this conjecture under the additional assumption that the group of contact automorphisms is reductive
and that the contact line bundle $L$ has ``enough'' sections. \\
If $b_2(X) \geq 2$ and $X$ is projective, then, due to \cite{KPSW00} and \cite{De02}, $X$ is a projectivized tangent bundle $\bP(T_Y)$ (in the sense of Grothendieck, taking hyperplanes)
over a projective manifold $Y$ (and conversely every such projectivised tangent bundle carries a contact structure). 
If $X$ is only K\"ahler, the analogous conclusion is unknown. By \cite{De02}, the canonical bundle $K_X$ is still not pseudo-effective in the K\"ahler setting, but---unlike in the projective case---it is 
not known whether this implies uniruledness of $X.$ \\
If however $X$ has enough symmetries, then we are able to deal with this situation:
\begin{theorem} Let $X$ be an almost homogeneous compact K\"ahler manifold carrying a contact structure. 
If $b_2(X) \geq 2,$ then there is a compact K\"ahler manifold $Y$ such that $X \simeq \bP(T_Y).$ 
\end{theorem} 
Here a manifold is said to be almost homogeneous, if the group of holomorphic automorphisms acts with an open orbit. 
Equivalently, the holomorphic vector fields generate the tangent bundle $T_X$ at some (hence at the general) point. \\
In this setting it might be interesting to try to classify all compact almost homogeneous K\"ahler manifolds $X$ of the form
$X = \bP(T_Y).$ Section 4 studies this question in dimension $3.$ 
\vskip .2cm \noindent 
In the second part of the paper we treat the deformation problem for projective contact manifolds. 
We consider a family 
\[
\pi\colon \sX \to \Delta
\]
of projective manifolds over the 1-dimensional disc $\Delta \subset \bC.$ Suppose that all $X_t = \pi^{-1}(t)$  are contact for $t \ne 0$. 
Is then $X_0$ also a contact manifold? \\
Suppose first that $b_2(X_t) = 1.$ Then---as discussed above---$X_t$ should be 
homogeneous for $t  \ne 0.$ Assuming homogeneity, the
situation is well-understood by the work of Hwang and Mok. In fact, then $X_0$ is again homogeneous with one surprising $7$-dimensional
exception, discovered by Pasquier-Perrin \cite{PP10} and elaborated further by Hwang \cite{Hw10}. Therefore one has rigidity and the contact
structure survives unless the Pasquier-Perrin case happens, where the contact structure does not survive. 
We refer to \cite{Hw10} and the references given at the beginning of section 5.
Therefore---up to the homogeneity conjecture---the situation is well-under\-stood. \\
If $b_2(X_t) \geq 2$, the situation gets
even more difficult; so we will assume that $X_t$ is homogeneous for $t \ne 0.$ We give a short argument in sect.~2, showing that then $X_t$
is either $\bP(T_{\bP_n})$ or a product of a torus and $\bP_n$. Then we investigate the {\it global projective rigidity} of $\bP(T_{\bP_n})$:
\begin{theorem} 
Let $\pi\colon \sX \to \Delta $ be a projective family of compact manifolds. If $X_t \simeq \bP(T_{\bP_n})$ for $t  \ne 0 $,
then either $X_0 \simeq \bP(T_{\bP_n})$ or $X_0 \simeq \bP(V)$ with some unstable vector bundle $V$ on $\bP_n.$ 
\end{theorem} 
The assumption that $X_0$ is projective is indispensable for our proof. In general, $X_0$ is only Moishezon, and in particular methods from
Mori theory fail. In case $X_0$ is even assumed to be Fano, the theorem was proved by Wi\'sniewski \cite{Wi91}; in this case $X_0 \simeq \bP(T_{\bP_n}).$ 
The case $X_0 \simeq \bP(V)$ with an unstable bundle really occurs; we provide examples in all dimensions in section~\ref{sec:degen}. In this case $X_0$ is no longer a contact manifold.

In general, without homogeneity assumption, $X_t $ is the projectivisation of the tangent bundle of some projective variety $Y_t;$ here we have only some partial results,
see Proposition~\ref{partial}. If however $X_t$ is again homogeneous ($t  \ne 0$) and not the projectivization of the tangent bundle of a projective space, 
then $X_t $ is a product of a torus $A_t$ and a projective
space, and we obtain a rather clear picture, described in Section~\ref{sec:irr}. 

The work on the project was started in collaboration with Kristina Frantzen. We would like to heartily thank her for her
contributions to sections 2, 3 and 4.

\section{Homogeneous K\"ahler Contact Manifolds}

We first study homogeneous manifolds which are projectivized tangent bundles.

\begin{proposition} \label{hom} Let $Y$ be compact K\"ahler. Then $X = \bP(T_Y)$ is homogeneous if and only if $Y$ is a torus or
$Y = \bP_n.$ 
\end{proposition}

\begin{proof} One direction being clear, assume that $X$ is homogeneous; thus $Y$ is homogeneous, too. 
The theorem of Borel-Remmert  \cite{BR62} says that 
$$Y \cong A \times G/P$$ where $G/P$ is a rational homogeneous manifold ($G$ a semi-simple complex Lie group and $P$ a parabolic subgroup) and $A$ a torus, one
factor possibly of dimension $0$. Let $d = \dim A \geq 0.$ \\
We first assume that $d > 0$. 
Then $T_Y = \sO_Y^d \oplus T_{G/P} $ leading to an inclusion
\[
Z := \bP(\sO_Y^d) \subset X
\]
with normal bundle 
\[
N_{Z/X} = \sO_Z(1) \otimes \pi^*q^*(\Omega^1_{G/P}) = p^*(\sO(1)) \otimes  \pi^*q^*(\Omega^1_{G/P}).
\]
Here $\pi\colon X \to Y$, $p\colon Z = \bP_{d-1} \times Y \to \bP_{d-1}$ and $q\colon Y \to G/P$ are the projections. 
Now, $X$ being homogeneous, $N_{Z/X}$ is spanned. This is only possible when $\dim G/P = 0$ so that $Y = A$. \\
If $d = 0$, then $X$ is rational homogeneous, hence Fano. This is to say that $T_Y$ is ample, hence
$Y = \bP_n$ (we do not need Mori's 
theorem here because $Y$ is already homogeneous). 
\end{proof} 

Proposition~\ref{hom} is now applied to obtain

\begin{proposition}
Let $X$ be a homogeneous compact K\"ahler manifold with contact structure and $\dim X = 2n-1$. Then either $X$ is a
Fano manifold (and therefore $X \simeq \bP(T_{\bP_n})$, by Prop.~\ref{hom}, unless $b_2(X) = 1$) or 
\[
X \cong A \times \mathbb P_{n-1} = \mathbb P (T_A),
\]
where $A$ denotes a complex torus of dimension $n$ and $T_A$ its holomorphic tangent bundle. 
\end{proposition}
\begin{proof}
Again by the theorem of Borel-Remmert, $X \cong A \times G/P$ where $G/P$ is rational-homogeneous and $A$ a torus, one
factor possibly of dimension~$0$. 
If $A$ does not appear, then $X$ is Fano with $b_2(X) \geq 2$ and therefore by \cite{KPSW00} of the form $X = \bP(T_Y)$. 
Then we conclude by Prop.~\ref{hom}. \\ 
So we may
assume $\dim A > 0$. 
Since a torus does not admit a contact structure, it follows that the factor $G/P$ is nontrivial, i.e.~$\dim G/P \geq 1$. 
We consider the projection $\pi\colon X \cong A \times G/P \to A $. Every fiber is $G/P$ and in particular a Fano manifold. 
We may therefore use the arguments of \cite{KPSW00}, Proposition 2.11, to conclude that every fiber is $\mathbb P_{n-1}$. 
Note that the arguments used in \cite{KPSW00}, Proposition 2.11 
do not use the assumption that $X$ is projective. This completes the proof. 
\end{proof}
\section{ The Almost homogeneous case}
The aim af this section is to generalize the previous section to almost homogeneous contact manifolds.
\subsection{Almost homogeneous projectivized tangent bundles}
We begin with the following general observation.
\begin{lemma}
 Let $Y$ be a compact complex manifold and let $X = \mathbb P(T_Y)$ be its projectivised tangent bundle. If $X$ is almost homogeneous, then $Y$ is almost homogeneous.
\end{lemma}
We already mentioned that if $X$ is homogeneous, so is $Y$.
\begin{proof}
Let $\pi\colon X \to Y$ be the bundle projection and consider the relative tangent sequence
\[
 0 \to T_{X/Y} \to T_X \to \pi^* T_Y \to 0.
\]
Since at a general point  of $X$ the tangent bundle $T_X$ is spanned by global sections, so is $\pi^* T_Y$. So if $y \in Y$ is general, if  $x \in \pi^{-1}(y)$ is general
and $v \in (\pi^*T_Y)_x $, then there exists 
$$s \in H^0(X,\pi^*(T_Y))$$
 such that $s(x) = v.$ Since $s = \pi^*(t)$ with
$t \in H^0(Y,T_Y),$ we obtain $t(y) = v \in T_{Y,y}.$ 
Thus $Y$ is almost homogeneous.
\end{proof}
\begin{remark} {\rm
Note that, conversely, the projectivized tangent bundle $X=\mathbb P (T_Y)$ of an almost homogeneous manifold $Y$ is in general {\bf not} almost homogeneous. This is illustrated by the following examples. }
\end{remark}
\begin{example}
We start in a quite general setting with a projective manifold $Y$ of dimension $n$. We assume that $Y$ is almost homogeneous with $h^0(Y,T_Y) = n$. Furthermore we assume 
\begin{equation} \label{eq1} h^0(Y,\Omega^1_Y \otimes T_Y) =  h^0(Y,\mathrm{End}(T_Y)) = 1,\end{equation} 
an assumption which is e.g.~satisfied if $T_Y$ is stable for some polarization.   
We let $X = \mathbb P (T_Y)$ be the projectivized tangent bundle with projection  $\pi\colon X = \mathbb P (T_Y)  \to Y$ and hyperplane bundle
$\sO_X(1).$
Pushing forward the relative Euler sequence to $Y$ yields  
\[
 0 \to \mathcal O_Y \to \Omega^1_Y \otimes \pi_* (\mathcal O_X(1)) \to \pi_* T_{X/Y} \to 0.
\]
Since $\pi_* (\mathcal O_X(1)) = T_Y$,  we obtain 
\[
 0 \to \mathcal O_Y \to \Omega^1_Y \otimes T_Y \to \pi_* T_{X/Y} \to 0.
\]
This sequence splits via the trace map $\Omega^1_Y\otimes T_Y\simeq\mathrm{End}(T_Y)\to\sO_Y$, so we obtain the exact sequence
\[
 0 \to H^0(Y,\mathcal O_Y) \to H^0(Y, \Omega^1_Y \otimes T_Y) \to H^0(Y,\pi_* T_{X/Y} ) \to 0.
\]
Using assumption \eqref{eq1} we find 
\[ H^0(X,T_{X/Y}) = 
 H^0(Y,\pi_* T_{X/Y}) = 0.
\]
Now the relative tangent sequence with respect to $\pi\colon X \to Y$ yields 
an exact sequence
\[
0 \to H^0(X,T_{X/Y}) \to H^0(X,T_X) \to H^0(X,\pi^*(T_Y)) \simeq H^0(Y,T_Y)
\]
and therefore
\[
h^0(T_X) \leq h^0(T_Y).
\]
Hence $h^0(T_X) \leq n$, and $X$ cannot be almost homogeneous. 

\vskip .2cm \noindent Notice that an inequality $h^0(T_X) \leq 2n-2$ suffices to conclude that $X$ is not almost homogeneous. Therefore we could weaken the assumptions $h^0(T_Y)=n$ and $h^0(\mathrm{End}(T_Y)) = 1$ to 
\[
h^0(T_Y)+h^0(\mathrm{End}(T_Y)) \leq 2n-2.
\]

\vskip .2cm \noindent We give two specific examples. 

\vskip .2cm \noindent First, 
let $Y$ be a del Pezzo surface of degree six, i.e., a three-point blow-up of $\mathbb P_2$. Its automorphisms group is $(\mathbb C^*)^2 \rtimes S_3$. 
In particular, $Y$ is almost homogeneous and $h^0(T_Y)=2$.
Since $h^0(\mathrm{End}(T_{\mathbb P_2}))=1$ and $Y$ is a blow up of $\mathbb P_2$, each endomorphism of $T_Y$ induces an endomorphism of $T_{\mathbb P_2}$ and it follows that 
\begin{equation}
h^0(T_Y \otimes \Omega^1_Y) = h^0(\mathrm{End}(T_Y)) = 1.
\end{equation}
Hence the assumptions of our previous considerations are fulfilled and $X = \bP(T_Y)$ is not almost homogeneous. 

\vskip .2cm \noindent
Here is an example with $b_2(Y) = 1.$ We let $Y$ be the Mukai-Umemura Fano threefold of type $V_{22},$ \cite{MU83}. 
Here $h^0(T_Y) = 3$ and $Y$ is almost homogeneous with ${\rm Aut}^0(Y) = {\rm SL}_2(\mathbb C).$  
Since $T_Y$ is known to be stable (see e.g.~\cite{PW95}), again all assumptions are satisfied and $X = \bP(T_Y)$ is not almost homogeneous.
\end{example}

\subsection{The Albanese map for almost homogeneous manifolds} 

A well-known theorem of Barth-Oeljeklaus determines the structure of the Albanese map of an almost homogeneous K\"ahler manifold.
\begin{theorem}[\cite{BO74}]
Let $X$ be an almost homogeneous compact K\"ahler manifold. Then the Albanese map $\alpha\colon X \to A$ is a fiber bundle. The fibers are connected, simply-connected and projective. 
\end{theorem}
\begin{remark}
 The fibers $X_a$ of $\alpha$ are almost homogeneous. 
\end{remark}
\begin{proof}
 Let $x,y \in X_a$ be two general points. Then there exists $ f \in \mathrm{Aut}(X)$ with $f(x)=y$. Since the automorphism $f$ is fiber preserving,
 we obtain an automorphism of $X_a$ mapping $x$ to $y$.  
\end{proof}
\subsection{The case $q(X) =0$}
If the irregularity of $X$ is $q(X)=0$, the Albanese map is trivial, and it follows that $X$ itself is simply-connected and projective. 
\begin{lemma}
Let $X$ be an almost homogeneous compact K\"ahler manifold with contact structure. If $q(X)=0$ and  $b_2(X) \geq 2$, then $X\cong \mathbb P(T_Y)$ is a projectivised tangent bundle. 
\end{lemma}
\begin{proof}
$X$ being projective, the results of \cite{KPSW00} apply. Combining them with \cite{De02} (cf. Corollary 4) yields the desired result.
\end{proof}
\begin{remark} \label{bea}
{\rm  The case where $q(X)=0$ and $b_2(X) =1$ remains to be studied. Here $X$ is an almost homogeneous Fano manifold. It would be interesting to find out whether the results of \cite{Be98} apply. I.e., 
one has to check whether $\mathrm{Aut}(X)$ is reductive and 
whether the map associated with the contact line bundle $L$ is generically finite. 

In order to study the second property, consider the long exact sequence
\[
 0 \to H^0(X,F) \to H^0(X,T_X) \to H^0(X,L) \to \dots
\]
If $H^0(X,F) \neq 0$ then $X$ has more than one contact structure \cite{Le95}, Prop.2.2, hence Corollary 4.5 of \cite{Ke01} implies that $X \cong \mathbb P_{2n+1}$ or $X \cong \mathbb P(T_Y)$.

If $H^0(X,F) = 0$ then $L$ has ``many sections'' and the map associated with $L$ is expected to be generically finite. }
\end{remark}
\subsection{The case $q(X) \geq 1$}
If the irregularity of $X$ is positive, then the Albanese map $\alpha\colon X \to A$ is a fiber bundle. We denote its fiber by $X_a$. 
\begin{lemma}
Let $X$ be an almost homogeneous compact K\"ahler manifold with contact structure and $q(X) \geq 1$. If the fiber $X_a$ of the Albanese map fulfills $b_2(X_a) =1$, then $X \cong \mathbb P (T_A)= \mathbb P_n \times A$, where $A$ is the Albanese torus of~$X$.
\end{lemma}
\begin{proof}
Since $b_2(X_a) =1$, then $X_a$ (being uniruled) is a Fano manifold. We may therefore apply Proposition 2.11 of \cite{KPSW00} (which works perfectly in our situation) to conclude 
that $\alpha\colon X \to A$ is a $\mathbb P_n$-bundle. The proof of Theorem 2.12 in \cite{KPSW00} can now be adapted to conclude that $X \cong \mathbb P (T_A)$. To be more specific, we already know in our
situation that $X = \bP(\sE)$ with $\sE = \alpha_*(L).$ The only thing to be verified is the isomorphism $\sE \simeq T_A.$ But this is seen as in the last part of the proof of Theorem 2.12 in
\cite{KPSW00}, since section 2.1 of \cite{KPSW00} works on any manifold. \\
So $X \simeq \bP(T_A)$ and $X \cong \mathbb P_n \times A$.  
\end{proof}

It remains to study the case where the fiber $X_a$ fulfills  $b_2(X_a) \geq 2$. In this case we consider a relative Mori contraction (over $A$; the projection is a projective morphism, \cite{Na87}, (4.12))
\[
\varphi\colon X \to Y.
\]
\begin{lemma}
We have $\dim X > \dim Y$. 
\end{lemma}
\begin{proof}
 The lemma follows from the fact that the restriction map $\varphi_a = \varphi \vert X_a$ is not birational. This can be shown by the same arguments as in Lemma 2.10 of \cite{KPSW00} using the length of the contraction and 
the restriction of the contact line bundle to the fiber $X_a$. Again the projectivity of $X$ is not needed in Lemma 2.10. 
\end{proof}
As above, we may now apply Proposition 2.11 of \cite{KPSW00} and conclude that the general fiber of $\varphi$ is $\mathbb P_n$. It remains to check that $\varphi$ is a $\mathbb P_n$-bundle and $X \cong \mathbb P(T_Y)$.
This is done again as in Theorem 2.12 of \cite{KPSW00} with Fujita's result generalized to the K\"ahler setting by Lemma \ref{Kaehler}.
Also the compactness assumption in \cite{Fu85} is not necessary, this will be important later.
\begin{lemma} \label{Kaehler} Let $X$ be a complex manifold, $f\colon X \to S $ a proper surjective map to a normal complex space $S.$ Let $L$ be a relatively ample line bundle on $X$ such that
$(F,L_F) \simeq (\bP_r,\sO(1))$ for a general fiber $F$ of $f.$ If $f$ is equidimensional, then $f$ is a $\bP_r$-bundle. 
\end{lemma} 

In total, we obtain

\begin{theorem} Let $X$ be a compact almost homogeneous K\"ahler contact manifold, $b_2(X) \geq 2.$ Then $X = \bP(T_Y)$ with a compact K\"ahler
manifold $Y.$
\end{theorem} 

The arguments above actually also show the following.

\begin{theorem} Let $X$ be a compact K\"ahler contact manifold. Let $\phi\colon X \to Y$ be a surjective map with connected fibers
such that $-K_X$ is $\phi$-ample and such that $\rho(X/Y) = 1 $ (we do not require the normal variety $Y$ to be K\"ahler). 
Then $Y$ is smooth and $X = \bP(T_Y).$ 
\end{theorem} 

One might wonder whether this is still true when $X$ is Moishezon or bimeromorphic to a K\"ahler manifold. Although there is
no apparent reason why the theorem should not hold in this context, at least the methods of proof completely fail. 
More generally, also the assumption that $X$ is almost homogeneous should be unnecessary. If $X$ is still K\"ahler, a Mori theory in the
non-algebraic case seems unavoidable. Already the question whether $X$ is uniruled is hard.

\subsection{Conclusion and open questions}
(1) In all but one case we find that a compact almost homogeneous K\"ahler contact manifold
 $X$ has the structure of a projectivised tangent bundle. The remaining case 
where $q(X)=0$ and $b_2(X) =1$ is discussed in Remark~\ref{bea}.  \\
(2) Can one classify all $Y$ (necessarily almost homogeneous) such that $\mathbb P(T_Y)$ is almost homogeneous? 
The case where $\dim Y= 2$ will be treated in the next section. One might also expect that if 
$Y = G/P$, then $X$  should be almost homogeneous. In case $Y$ is a Grassmannian or a quadric, this has been checked by Goldstein \cite{Go83}. 
Of course, if $Y = \bP_n,$ then $X$ is even homogeneous. 

\section{Almost homogeneous contact threefolds} 

In this section we specialize to almost homogeneous contact manifolds in dimension $3$. 

\begin{theorem} \label{list}

Let $X$ be a smooth compact K\"ahler threefold which is of the form $X = \bP(T_Y)$ for some compact (K\"ahler)
surface $Y$. 
\begin{enumerate} 
\item If $X$ is almost homogeneous, then
$Y$ is a minimal surface or a blow-up of $\bP_2$ or  $Y = \bF_n = \bP(\sO_{\bP_1} \oplus \sO_{\bP_1}(-n))$ for some $n \geq 0$, $n\ne 1$
\item If $Y$ is minimal, then $X$ is almost homogeneous if and only if $Y$ is one of the following surfaces.
\begin{itemize} 
\item $Y = \bP_2$
\item $Y = \bF_n $ for some $n \geq 0$, $n\ne 1$
\item $Y$ is a torus
\item $Y = \bP(\sE)$ with $\sE$ a vector bundle of rank 2 over an elliptic curve which is either a direct sum 
of two topologically trivial line bundles or the non-split extension of two trivial line bundles.
\end{itemize} 
\end{enumerate} 

\end{theorem} 

\begin{proof} Suppose $X$ is almost homogeneous. Then $Y$ is almost homogeneous, too (Lemma 3.1). By Potters' classification 
\cite{Po68}, $Y$ is one of the following.
\begin{enumerate} 
\item $Y = \bP_2$
\item $Y = \bF_n = \bP(\sO_{\bP_1} \oplus \sO_{\bP_1}(-n))$ for some $n \geq 0$, $n\ne 1$
\item $Y$ is a torus
\item $Y = \bP(\sE)$ with $\sE$ a vector bundle of rank 2 over an elliptic curve which is either a direct sum 
of two topologically trivial line bundles or the non-split extension of two trivial line bundles
\item $Y$ is a certain blow-up of $\bP_2$ or of $\bF_n.$
\end{enumerate} 

This already shows the first claim of the theorem, and
it suffices to assume $Y$ to be a minimal surface of the list and to check whether $X = \bP(T_Y)$ is almost homogeneous. 
In cases (1) and (3) this is clear; $X$ is even homogeneous. \\
To proceed further, consider the tangent bundle sequence
\[
0 \to T_{X/Y} \to T_X \to \pi^*(T_Y) \to 0.
\]
Notice
\[
h^0(T_{X/Y}) = h^0(-K_{X/Y}) = h^0(S^2T_Y \otimes K_Y).
\]
Applying $\pi_*$ and observing that the connecting morphism
\[
T_Y \to R^1\pi_*(T_{X/Y})
\]
(induced by the Kodaira-Spencer maps) vanishes since $\pi$ is locally trivial,
it follows that 
\[
H^0(X,T_X) \to H^0(X,\pi^*(T_Y)) = H^0(Y,T_Y)
\]
is surjective. If therefore
\begin{equation}\label{reltang}
H^0(X,T_{X/Y}) \simeq H^0(Y,S^2T_Y \otimes K_Y) \ne 0, \tag{$*$}
\end{equation}
the tangent bundle $T_X$ is obviously spanned and therefore $X$ is almost homogeneous. \\
In case (4), \eqref{reltang} is now easily verified: 
Let $p\colon \bP(\sE) \to C$ be the $\bP_1$-fibration over the elliptic curve $C$. The tangent bundle sequence reads
\[
0 \to -K_Y \to T_Y \to \sO_Y \to 0.
\]
Since $T_Y$ is generically spanned, the map $H^0(\sO_Y) \to H^1(-K_Y)$ must vanish, so that the sequence splits:
\[
T_Y \simeq -K_Y \oplus \sO_Y.
\]
Thus $S^2T_Y \otimes K_Y \simeq -K_Y \oplus \sO_Y \oplus K_Y $ and ($*$) follows. 
\vskip .2cm 
Now if $Y=\bF_n$ as in (2), let $p\colon Y\to\bP_1$ be the natural projection. The relative tangent sequence then reads
\begin{equation}\label{tangseq}
0 \to T_{Y/\bP_1}\to T_Y\to p^*\sO_{\bP_1}(2)\to0. \tag{$**$}
\end{equation}
Taking the second symmetric power and tensorizing with $K_Y$ yields
\[
0 \to T_Y \otimes T_{Y/\bP_1} \otimes K_Y \to S^2T_Y \otimes K_Y \to p^*\sO_{\bP_1}(4)\otimes K_Y \to 0,
\]
so, by~\eqref{tangseq}, we obtain an inclusion
\[
H^0(T_{Y/\bP_1}^{\otimes2}\otimes K_Y)\subset H^0(S^2T_Y\otimes K_Y).
\]
Now by the relative Euler sequence, $T_{Y/\bP_1} \simeq \sO_Y(2) \otimes p^*\sO_{\bP_1}(n)$, and thus
\[
H^0(T_{Y/\bP_1}^{\otimes2}\otimes K_Y) \simeq H^0(\sO_Y(2) \otimes p^*\sO_{\bP_1}(n-2)).
\]
Now since
\[
p_*(\sO_Y(2)\otimes p^*\sO_{\bP_1}(n-2)) \simeq \sO_{\bP_1}(n-2) \oplus \sO_{\bP_1}(-2) \oplus \sO_{\bP_1}(-n-2),
\]
we have shown \eqref{reltang} to be true for $n\ge2$. If $n=0$, i.e., $Y\simeq\bP_1\times\bP_1$, the sequence~\eqref{tangseq} splits and an easy calculation shows that~\eqref{reltang} is satisfied also in this case.
\end{proof} 

\begin{remark} {\rm  The case that $Y$ is a non-minimal rational surface in Theorem~\ref{list} could be further studied, but this is a rather tedious task. } 
\end{remark}

%%%%%%%%%%%%%%%%%%%%%%%%%%%%%%%%%%%%%%%%%%%%%%%%%%%%%%%%%%%%%%%%%%%% 

\section {Deformations I: the rational case}

We consider a family $\pi\colon \mathcal X \to \Delta $ of compact manifolds over the unit disc $\Delta \subset \KC$. As usual, we let $X_t = \pi^{-1}(t).$ 
We shall assume $X_t$ to be a projective manifold for {\it}all $t$, so we are only interested in projective families here. If now $X_t$ is a contact 
manifold for $t \ne 0,$ when is $X_0$ still a
contact manifold? \\
If $b_2(X_t) = 1,$ there is a counterexample due to \cite{PP10}, see also \cite{Hw10}. Here the $X_t$ are $7$-dimen\-sional
rational-homogeneous contact manifolds and $X_0$ is a non-homogeous non-contact manifold. If one believes that any Fano 
contact manifold with $b_2 = 1$ is rational-homogeneous, then due to the results of Hwang and Mok, this is the only example
where a limit of contact manifolds with $b_2 = 1$ is not contact. \\
If $b_2(X_t) \geq 2$, it is no longer true that the limit $X_0$ is always a contact manifold, as can be seen from the following example: We let $\sY\to\Delta$ be a family of compact manifolds 
such that $Y_t\simeq\bP_1\times\bP_1$ for $t\ne0$ and $Y_0\simeq \bF_2$. Then there exist line bundles $\sL_1$ and $\sL_2$ on $\sY$ such that $\sL_1|Y_t\simeq\sO_{\bP_1\times\bP_1}(2,0)$ 
and $\sL_2|Y_t\simeq\sO_{\bP_1\times\bP_1}(0,2)$ for every $t\ne 0$. If we let $\sX:=\bP(\sL_1\oplus\sL_2)$, then $X_t\simeq\bP(T_{Y_t})$ for $t\ne 0$, but $X_0\not\simeq\bP(T_{Y_0})$. \\
However $\bP(T_{\bP_1 \times \bP_1}) $ is not homogeneous; in fact by Proposition~\ref{hom}, $ \bP(T_{\bP_n})$  is the only homogeneous rational contact manifold with $b_2 \geq 2.$ In this prominent case 
we prove global projective rigidity, i.e., $X_0 =  \bP(T_{\bP_n})$, unless $X_0$ is the projectivization of some unstable bundle, so that 
both contact structures survive in the limit. In the ``unstable case'', the contact structure does not survive. The special 
case where $X_0$ is Fano is due to Wi\'sniewski \cite{Wi91}; here global rigiditiy always holds. 
\vskip .2cm There is a slightly different point of view, asking whether projective limits of rational-homoge\-neous manifolds are
again rational-homogeneous. As before, if $b_2(X_t) = 1, $ this is true by the results of Hwang and Mok with the $7$-dimensional
exception. In case $b_2(X_t) \geq 2$, this is false in general (e.g.~for $\bP_1 \times \bP_1$), but the picture under which circumstances 
global rigidity is still true is completely open.

%%%%%%%%%%%%%%%%%%%%%%%%%%%%%%%%%%%%%%%%

\begin{theorem} \label{uniruled} Let $\pi\colon \mathcal X \to \Delta $ be a family of compact manifolds. Assume $X_t \simeq
\bP(T_{\bP_n})$ for $t \ne 0.$ If $X_0$ is projective, then either $X_0 \simeq  \bP(T_{\bP_n})$ or $X_0 \simeq \bP(V)$ with some unstable vector
bundle $V$ on $\bP_n.$ 
\end{theorem} 

\begin{proof} 
Since $K_{\mathcal X}$ is not $\pi$-nef, there exists a relative Mori contraction (see \cite{Na87}, (4.12), we may shrink $\Delta$)
\[
\Phi\colon \mathcal X \to \mathcal Y
\]
over $\Delta$. 
Put $\Delta^* = \Delta \setminus \{0\}$
and $\mathcal X^* = \mathcal X \setminus X_0; \ \mathcal Y^* = \mathcal Y \setminus Y_0$.
Now $\phi_t = \Phi \vert X_t$ is a Mori contraction for any $t$ (cp.~\cite{KM92}, (12.3.4), but this is pretty clear in our simple situation), {\em unless possibly $\phi_t$ is biholomorphic}
for $t \ne 0$.

Now since $\sX$, $\Delta$ and $\pi$ are smooth, the latter case cannot occur by~\cite{Wi91b}, (1.3), so $\phi_t$ is the contraction of an extremal ray for any $t\in\Delta$. 
Let $\tau\colon \mathcal Y \to \Delta$ be the induced projection and set $Y_t = \tau^{-1}(t), $ so that
$Y_t \simeq \bP_n$ for $t \ne 0$. Since $\mathcal Y$ is normal, the normal variety $Y_0$ must also have
dimension $n.$

\vskip .2cm \noindent From the exponential sequence, Hodge decomposition and the topological triviality of the family~$\sX,$ it follows
that 
\[
{\rm Pic}(\sX) \simeq H^2(\sX,\bZ) \simeq \bZ^2
\]
and that
\[
{\rm Pic}(X_0) \simeq H^2(X_0,\bZ^2) \simeq \bZ^2.
\]
Furthermore, the restriction ${\rm Pic}(\sX) \to {\rm Pic}(X_0) $ is bijective. As an immediate consequence, 
we can write
\[
-K_{\mathcal X} = n \sH
\]
with a line bundle $\sH$ on $\mathcal X.$ Let $\sH_t = \sH \vert X_t$ 
so that $\sH_t \simeq \sO_{\bP(T_{\bP_n})}(1)$ for $t \ne 0.$ 

%\vskip .2cm \noindent
\begin{claim} $Y_0 \simeq \bP_n.$
\end{claim} 
In fact, by our previous considerations, there is a unique line bundle $\sL$ on $\mathcal X$ such that
\[
\sL \vert X_t = \phi_t^*(\sO_{\bP_n}(1))
\]
for $t \ne 0.$ Moreover $\sL \vert X_0 = \phi_0^*(\sL')$ with some ample line bundle $\sL'$ on $Y_0.$ 
Therefore by semi-continuity,
\[
h^0(\sL') = h^0(\sL \vert X_0) \geq n+1
\]
and 
\[
c_1(\sL')^n = 1.
\]
Hence by results of Fujita \cite{Fu90}, (I.1.1), see also \cite{BS95}, (III.3.1), we have $(Y_0,\sL') \simeq (\bP_n,\sO(1)).$

In particular we obtain

\begin{subcorollary} $\mathcal Y $ is smooth and $\mathcal Y \simeq \bP_n \times \Delta$. 
\end{subcorollary} 

\vskip .2cm \noindent
Next we notice that the general fiber of $\phi_0$ must be $\bP_{n-1}$, since it is a smooth degeneration of fibers
of $\phi_t$ (by the classical theorem of Hir\-ze\-bruch--Ko\-dai\-ra). \\
One main difficulty is that $\phi_0$ might not be equidimensional. If we know equidimensionality, we may apply [Fu85, 2.12] to conclude that
$X_0 = \bP(\sE_0)$ with a locally free sheaf $\sE_0$ on $Y_0.$  \\
We introduce the torsion free sheaf
\[
\sF = \Phi_*(\sH) \otimes \sO_{\mathcal Y}(-1).
\]
Since  
\[
\mathop{\rm codim} \Phi^{-1}({\rm Sing}(\sF)) \geq 2,
\]
the sheaf $\sF$ is actually reflexive and of course locally free outside $Y_0.$ 
In the following Sublemma we will prove that $\sF$ is actually locally free. 

\begin{sublemma} \label{sub1} $\sF$ is locally free and therefore $\sX = \bP(\sF). $
\end{sublemma} 

\begin{proof}
As explained above, it is sufficient to show that
\[
\phi_0\colon X_0 \to \bP_n
\]
is equidimensional. So let $F_0$ be an irreducible component of a fiber of~$\phi_0$. Then $F_0$ gives rise to a class
\[
[F_0] \in H^{2k}(X_0,\bQ),
\]
where we denote by~$k$ the codimension of~$F_0$ in~$X$. Obviously $k\le n$, and we must exclude the case that $k<n$.

So we assume in the following that $k<n$. Then, since $X_0$ is homeomorphic to $\bP(T_{\bP_n})$, the Leray--Hirsch theorem gives
\[
\dim H^{2k}(X_0,\bQ) = k + 1.
\]
Now if we denote by $H$ the class of a hyperplane in $\bP_n$, and by $L$ the class of an ample divisor on~$X_0$, then the classes
\begin{equation}\label{cohbasis}
L^k, L^{k-1}.(\phi_0^* H), \dots, L.(\phi_0^* H)^{k-1}, (\phi_0^* H)^k 
\end{equation}
form a basis of $H^{2k}(X_0,\bQ)$, which can be seen as follows: By the dimension formula given above, it is sufficient to show linear independency, so assume that we are given $\lambda_0$, $\dots$,  $\lambda_k\in\bQ$ such that
\begin{equation}\label{linunab}
\sum_{\ell=0}^k\lambda_\ell L^{k-\ell}.(\phi_0^*H)^\ell =0.
\end{equation}
Now let $\ell_0 \in \{0,\dots,k\}$. By induction, we assume that $\lambda_{\ell}=0$ for all $\ell<\ell_0$. Then intersecting \eqref{linunab} with
$L^{n-k-1+\ell_0}.(\phi_0^*H)^{n-\ell_0}$ yields
\[\lambda_{\ell_0} L^{n-1}.(\phi_0^*H)^{n} = 0,\]
thus $\lambda_{\ell_0} = 0$ since $L^{n-1}.(\phi_0^*H)^n >0$.

So \eqref{cohbasis} is indeed a basis of~$H^2(X_0,\bQ)$ and we can write
\begin{equation}\label{linkomb}
[F_0] = \sum_{\ell=0}^k\alpha_\ell L^{k-\ell}.(\phi_0^*H)^\ell 
\end{equation}
for some $\alpha_0$, $\dots$, $\alpha_k\in\bQ$. We now let $\ell_0\in\{0,\dots,k\}$ and assume that $\alpha_\ell=0$ for $\ell<\ell_0$. We observe that $[F_0].(\phi_0^*H)^{n-\ell_0}=0$ 
since $F_0$ is contained in a fiber of~$\phi_0$ and $\ell_0\le k < n$. Hence, intersecting \eqref{linkomb} with $L^{n-k-1+\ell_0}.(\phi_0^*H)^{n-\ell_0}$ yields
\[
0 = \alpha_{\ell_0}L^{n-1}.(\phi_0^*H)^n,
\]
so we deduce $\alpha_{\ell_0}=0$ as before. Therefore by induction, we have $[F_0]=0$, which is impossible, $X_0$ being projective.
\end{proof}

Now we set $V = \sF \vert X_0.$ If the bundle $V$ is semi-stable, then $V \simeq T_{\bP_n}$ and the theorem is settled.
\end{proof}

%%%%%%%%%%%%%%%%%%%%%%%%%%%%%%%%%%%%%%%%%%%

Suppose in Theorem~\ref{uniruled} that $X_0 \simeq \bP(V)$ with an unstable bundle $V$ (we will show in section~\ref{sec:degen} that this can indeed occur).   
Then $X_0$ does not carry a contact structure. In fact, otherwise $X_0 \simeq \bP(T_S)$ with some projective variety $S$, \cite{KPSW00}. 
Hence $X_0$ has two extremal contractions, and therefore $X_0$ is Fano. Hence $T_S$ is ample and thus $S \simeq \bP_n$ (or apply Wi\'sniewski's theorem).  
Therefore we may state the following 

\begin{corollary} 
Let $\pi\colon \mathcal X \to \Delta $ be a family of compact manifolds. Assume $X_t \simeq
\bP(T_{\bP_n})$ for $t \ne 0.$ If $X_0$ is a projective contact manifold, then $X_0 \simeq  \bP(T_{\bP_n})$. 
\end{corollary}

In the situation of Theorem~\ref{uniruled}, we had two contact structures on $X_t.$ 
This phenomenon is quite unique because of the following result \cite{KPSW00}, Prop. 2.13. 

\begin{theorem} \label{unique} Let $X$ be a projective contact manifold of dimension $2n-1$ admitting two extremal rays in the 
cone of curves
$\overline{NE}(X)$.
Then
$X \simeq \bP(T_{\bP_n})$. 
\end{theorem}

Here is an extension of Theorem~\ref{unique} to the non-algebraic case.

\begin{theorem} Let $X$ be a compact contact K\"ahler manifold admitting two contractions $\phi_i\colon X \to Y_i$ to normal compact
K\"ahler spaces $Y_i$. This is to say
that $-K_X $ is $\phi_i$-ample and that $\rho(X/Y_i) = 1$. 
Then $X$ is projective and therefore $X = \bP(T_{\bP_n})$. 
\end{theorem} 

\begin{proof} We already know by Theorem 3.13 that $X = \bP(T_{Y_i}).$ Let $F \simeq \bP_{n-1}$ be a fiber of $\phi_2$. 
Then the restriction $\phi_1 \vert F$ is finite. We claim that $Y_1$ must be projective. In fact, consider the rational quotient, 
say $f\colon Y_1 \dasharrow Z$, which is an almost holomorphic map to a compact K\"ahler manifold $Z.$ By construction, the map $f$ contracts the images
$\phi_1(F)$, hence $\dim Z \leq 1$. But then $Z$ is projective and therefore $Y_1$ is projective, too (e.g. by arguing that $y_1$ cannot carry a
holomorphic $2$-form). \\
By symmetry, $Y_2$ is projective, too. Since the morphisms $\phi_i$ induce a finite map $X \to Y_1 \times Y_2 $ (onto the image of $X$), the
variety $X$ is also projective.
\end{proof}

Any projective contact manifold $X$ with $b_2(X) \geq 2$ is of the form $X = \bP(T_Y)$. Therefore it is natural ask for generalizations of Theorem \ref{uniruled},
substituting the projective space by other projective varieties. 

\begin{proposition} \label{partial} Let $\pi\colon \sX \to \Delta$ be a projective family of compact manifolds $X_t$ of dimension $2n-1.$
Assume that $X_t  \simeq \bP(T_{Y_t}) $ for $t  \ne 0$ with (necessarily projective) manifolds $Y_t \ne \bP_n.$ Assume that
$H^q(X_t,\sO_{X_t}) = 0$ for $q = 1,2$ for some (hence all) $t$.
Then the following statements hold. 
\begin{enumerate}
\item There exists a relative 
contraction $\Phi\colon \sX \to \sY$ over $\Delta$ such that $\Phi \vert X_t$ is the given 
$\bP_{n-1}$-bundle structure for $t \ne 0$.
\item If $\phi_0 := \Phi \vert X_0$ is equidimensional, then  $X_0 \simeq \bP(\sE_0) $ with a rank-$n$ bundle $\sE$ over the projective manifold $Y_0$; 
and $Y_0$ is the limit
manifold of a family $\tau\colon \sY \to \Delta$ such that $Y_t \simeq \tau^{-1}(t) $ for $t  \ne 0$. In other words, 
$\sX \simeq \bP(\sE)$ such that $\sE = T_{\sY/\Delta} $ over $\Delta \setminus\{0\}$.
%\item If $\dim X_0 \leq 5, $ then $\phi_0$ is equidimensional.
%\item ????? If $T_{Z_0}$ is semi-stable for some polarization, then there exists 
%a relative 
%contraction $\Phi: \sX \to \sY$ over $\Delta$ with a submersion $\sY \to \Delta$ such that $\sX = \bP(T_{\sY/\Delta}).$ 
\end{enumerate} 
\end{proposition}

\begin{proof}
Since $Y_t \ne \bP_n$ by assumption, every $X_t$, $t  \ne 0$, has a unique Mori contraction, 
the projection $\psi_t\colon X_t \to Y_t$, by Theorem~\ref{unique}.  
As in the proof of Theorem~\ref{uniruled}, we obtain a
relative Mori contraction 
\[
\Phi\colon \sX \to \sY
\]
over $\Delta$, and necessarily 
$\Phi \vert X_t = \phi_t$ for all $t  \ne 0$ (we use again~\cite{Wi91b}, (1.3)). This already shows Claim (1).
\vskip .2cm \noindent 
If $\phi_0$ is equidimensional, we apply---as in the proof of Theorem \ref{uniruled}---\cite{BS95}, (III.3.2.1), to conclude that 
there exists a locally free sheaf $\sE_0$ of rank $n$ on $Y_0$ such that 
$X_0 \simeq \bP(\sE_0)$, proving (2).
\end{proof} 

\begin{theorem}  Let $\pi\colon \sX \to \Delta$ be a projective family of compact manifolds $X_t$ of dimension $2n-1.$
Assume that $X_t  \simeq \bP(T_{Y_t}) $ for $t  \ne 0$ with (necessarily projective) manifolds $Y_t (\ne \bP_n).$ Assume that
$H^q(X_t,\sO_{X_t}) = 0$ for $q = 1,2$ for some (hence all) $t$. Assume moreover that 
\begin{enumerate} 
\item $\dim X_0 \leq 5,$ or
\item $b_{2j}(Y_t) = 1 $ for some $t \ne 0$ and all $1 \leq j < {n \over {2}}$. 
\end{enumerate} 
Then there exists a  relative 
contraction $\Phi\colon \sX \to \sY$ over $\Delta$ such that $\Phi \vert X_t$ is the given 
$\bP_{n-1}$-bundle structure for $t \ne 0$. Moreover there is a locally free sheaf $\sE$ on~$\sY$ such that $\sX \simeq \bP(\sE)$ and $\sE \vert Y_t \simeq T_{Y-t}$ for all $t \ne 0$. 
\end{theorem} 

\begin{proof} By the previous proposition it suffices to show that  $\phi_0 = \Phi \vert X_0$ is equidimensional. \\
(1) First suppose that $\dim X_0 \leq 5$.  Then $1 \leq \dim Y_0  \leq 3$. The case $\dim Y_0 = 1$ is trivial. If $\dim Y_0 = 2$, then
all fibers must have codimension 2, because $\phi_0$ does not contract a divisor  (the relative Picard number being $1$).  
If $\dim Y_0 = 3$, then by \cite{AW97}, (5.1), we cannot have a 3-dimensional fiber. Since again there is no 4-dimensional fiber, $\phi_0$ must be equidimensional also in this case. \\
(2) If  $b_{2j}(Y_t) = 1 $ for some $t$ and all $1 \leq j \leq {n \over {2}}$,  then $b_{2k}(X_t) = k+1$ for $k < n$ and  we may simply argue as in Sublemma~\ref{sub1} to conclude that $\phi_0$ is equidimensional 
(the smoothness of $Y_0$ is not essential in the reasoning of Sublemma~\ref{sub1}). 
\end{proof}

%%%%%%%%%%%%%%%%%%%%%%%%%%%%%%%%%%%%%%%%%%%%%%%%%%%%%%%%%%%%%%%%%%%%%%%%%%%%%%%%%%%%%%%%%

\section{Degenerations of $T_{\bP_n}$}
\label{sec:degen}

In view of Theorem~\ref{uniruled}, we can ask the question which bundles can occur as degenerations of~$T_{\bP_n}$, i.e., for which rank-$n$ bundles $V$ on~$\bP_n$ there exists a rank-$n$ bundle $\sV$ on $\bP_n\times\Delta$ such that
\[
\sV_t:=\sV|{\bP_n\times\{t\}}\simeq \begin{cases}
T_{\bP_n},&\text{for $t\ne0$},\\
V, &\text{for $t=0$}.
\end{cases}
\]

In the case that~$n\ge3$ is odd, it was already observed by Hwang in~\cite{Hw06} that one can easily construct a nontrivial degeneration of~$T_{\bP_n}$ as follows: We consider the null correlation bundle on $\bP_n$, which is a rank-$(n-1)$ bundle $N$ on~$\bP_n$ given by a short exact sequence
\[
0 \longrightarrow N \longrightarrow T_{\bP_n}(-1) \longrightarrow \sO_{\bP_n}(1) \longrightarrow 0.
\]
(cf.~\cite{OSS80}, (I.4.2)).
The existence of this sequence now implies that $T_{\bP_n}$ can be degenerated to~$N(1)\oplus\sO_{\bP_n}(2)$.

When $n$ is even, matters become more complicated, but we can still obtain nontrivial degenerations:
\begin{proposition}
Let $n\ge 2$. Then there exists a rank-$n$ bundle $\sV$ on $\bP_n\times\Delta$ such that $\sV_t\simeq T_{\bP_n}$ for $t\ne 0$ and $h^0(\sV_0(-2))=1$, so in particular $\sV_0\not\simeq T_{\bP_n}$.
\end{proposition}
\begin{proof}
We construct an inclusion of vector bundles
\[
A\colon \Omega^1_{\bP_n\times\Delta/\Delta}(2)\oplus \sO_{\bP_n\times\Delta}
\hookrightarrow \sO_{\bP_n\times\Delta}(1)^{\oplus(n+1)}\oplus \Omega^1_{\bP_n\times\Delta/\Delta}(2)
\]
via a family $A=(A_t)_{t\in\Delta}$ of matrices
\[
A_t = 
\begin{pmatrix}
\alpha_t & \beta_t \\
\gamma_t & \delta_t
\end{pmatrix}
\]
of sheaf homomorphisms
\begin{align*}
\alpha_t&\colon \Omega^1_{\bP_n}(2) \to \sO_{\bP_n}(1)^{\oplus(n+1)},
&\beta_t&\colon \sO_{\bP_n} \to \sO_{\bP_n}(1)^{\oplus(n+1)},\\
\gamma_t&\colon \Omega^1_{\bP_n}(2) \to \Omega^1_{\bP_n}(2),
&\delta_t&\colon \sO_{\bP_n} \to \Omega^1_{\bP_n}(2),
\end{align*}
which we define as follows: We take $\alpha_t$ and $\beta_t$ to be the natural inclusions coming from the Euler sequence and its dual, where we choose the coordinates on~$\bP_n$ such that
\[
\beta_t(\sO_{\bP_n}) \not\subset \alpha_t(\Omega^1_{\bP_n}(2)).
\]
This implies that the map
\[
\alpha_t\oplus \beta_t\colon \Omega^1_{\bP_n}(2) \oplus \sO_{\bP_n} \to \sO_{\bP_n}(1)^{\oplus(n+1)}
\]
is generically surjective. Since $\Omega^1_{\bP_n}(2)\simeq\Lambda^{n-1}(T_{\bP_n}(-1))$ is globally generated, a general section in $H^0(\Omega^1_{\bP_n}(2))$ has 
only finitely many zeroes. Since $\Omega^1_{\bP_n}(2)$ is homogeneous, we can thus choose the map $\delta_t$ in such a way that its zeroes are disjoint from the locus where $\alpha_t\oplus\beta_t$ 
is not surjective. Finally we let $\gamma_t=t\cdot\id$.

Now in order to show that $A$ is an inclusion of vector bundles, we need to show that for any point $(p,t)\in\bP_n\times\Delta$, the matrix
\[
A_t(p) = 
\begin{pmatrix}
\alpha_t(p) & \beta_t(p) \\
\gamma_t(p) & \delta_t(p)
\end{pmatrix}
\in \bC^{(2n+1)\times(n+1)}
\]
has rank~$n+1$. For semicontinuity reasons, shrinking $\Delta$ if necessary, we can assume $t=0$, then the rank condition follows easily from the choice of $\alpha_0$, $\beta_0$, $\gamma_0$, $\delta_0$.

We now let
\[
\sV := \coker A.
\]
It remains to investigate the properties of the bundles $\sV_t:=\sV|\bP_n\times\{t\}$. For each $t\in\Delta$, we have an exact sequence of vector bundles
\begin{equation}\label{monad}
0\longrightarrow \Omega^1_{\bP_n}(2)\oplus \sO_{\bP_n}
\longrightarrow \sO_{\bP_n}(1)^{\oplus(n+1)}\oplus \Omega^1_{\bP_n}(2)\longrightarrow\sV_t\longrightarrow 0.
\end{equation}
We want to calculate $H^q(\sV_t(-1-k))$ for $k=0$, $\dots$, $n$. From the Bott formula we obtain for $(k,q)\in\{0,\dots,n\}^2$:
\[
h^q(\Omega^1_{\bP_n}(1-k))=\begin{cases}
1, & \text{for $(k,q)=(1,1)$},\\
0, &\text{otherwise}.
\end{cases}
\]
Now if we tensorize \eqref{monad} with~$\sO_{\bP_n}(-1-k)$, take the long exact cohomology sequence and observe that $H^q(\delta_0)=0$ for every~$q$, we get for $(k,q)\in\{0,\dots,n\}^2$:
\[
h^q(\sV_0(-1-k)) = \begin{cases}
n+1, &\text{for $(k,q)=(0,0)$},\\
1, &\text{for $(k,q)\in\{(1,0),(1,1),(n,n-1)\}$},\\
0, &\text{otherwise}.
\end{cases}
\]
Similarly, if we observe that $H^q(\delta_t)=\id$ for $t\ne0$, we obtain for $t\ne0$, $(k,q)\in\{0,\dots,n\}^2$:
\[
h^q(\sV_0(-1-k)) = \begin{cases}
n+1, &\text{for $(k,q)=(0,0)$},\\
1, &\text{for $(k,q)\in\{(n,n-1)\}$},\\
0, &\text{otherwise}.
\end{cases}
\]
The proposition now follows from Lemma~\ref{beillem}.
\end{proof}
\begin{lemma}\label{beillem}
Let $V$ be a vector bundle on~$\bP_n$ such that for any $(k,q)\in\{0,\dots,n\}^2$, we have
\[
h^q(V(-1-k)) = \begin{cases}
n+1, &\text{for $(k,q)=(0,0)$},\\
1, &\text{for $(k,q)=(n,n-1)$},\\
0, &\text{otherwise}.
\end{cases}
\]
Then $V\simeq T_{\bP_n}$.
\end{lemma}
\begin{proof}
We consider the Beilinson spectral sequence for the bundle $V(-1)$, which has $E_1$-term
\[
E_1^{pq} = H^q(V(-1+p))\otimes\Omega^{-p}_{\bP_n}(-p)
\]
(cf.~\cite{OSS80}, (II.3.1.3)).

By assumption, $E_1^{pq}=0$ for $(p,q)\not\in\{(0,0),(-n,n-1)\}$ and
\[
E_1^{0,0} = \sO_{\bP_n}^{\oplus(n+1)}, \quad E_1^{-n,n-1} = \sO_{\bP_n}(-1).
\]
In particular, the only nonzero differential occurs at the $E_n$-term, namely
\[d_n^{-n,n-1}\colon E_n^{-n,n-1}\to E_n^{0,0}.\]
Since $E_\infty^{pq}=0$ for $p+q\ne0$ and $E_\infty^{-p,p}$ are the quotients of a filtration of~$V(-1)$, the differential $d_n^{-n,n-1}$ induces a short exact sequence
\begin{equation} \label{beilinsoneuler}
0 \longrightarrow \sO_{\bP_n}(-1) \stackrel{d_n^{-n,n-1}}{\longrightarrow} \sO_{\bP_n}^{\oplus(n+1)} \longrightarrow V(-1) \longrightarrow 0.
\end{equation}
Now since $V$ is locally free, the map $d_n^{-n,n-1}$ cannot have zeroes, so \eqref{beilinsoneuler} must be an Euler sequence, whence $V(-1)\simeq T_{\bP_n}(-1)$.
\end{proof}

\section{Deformations II: positive irregularity}
\label{sec:irr}

A homogeneous compact contact K\"ahler manifold $X$ of dimension $2n+1$ with $b_2(X) \geq 2$ is either $\bP(T_{\bP_{n+1}})$ or 
a product $A \times \bP_n$ with a torus $A$ of dimension $n+1$. Here we study in general the K\"ahler deformations of $A \times \bP_n,$ where $A$ is
an $m$-dimensional torus. 

\begin{theorem}  \label{irreg} Let $\pi\colon \mathcal X \to \Delta $ be a family of compact manifolds over the unit disc
$\Delta \subset \mathbb C$. Assume $X_t \simeq A_t \times
\bP_n$ for $t \ne 0$, where $A_t$ is a torus of dimension $m.$ 
If $X_0$ is in class $\sC$, then the relative Albanese morphism realises $\sX$ as a submersion $\alpha\colon \sX \to \sA$, where 
$\pi\colon \sA \to \Delta$ is torus bundle such that $\pi^{-1}(t) \simeq A_t$ for $t \ne 0.$ Moreover there is a locally free sheaf
$\sE$ over $\sA$ such that $\sX = \bP(\sE)$, $\sX_t \simeq \bP(\sE_t)$ for all $t$ and $\sE \vert A_t \simeq \sO_{A_t}^{n+1}$ for $t \ne 0.$
\end{theorem} 

\begin{proof} Let $m = \dim A_t = q(X_t)$ for $t \ne 0.$ Hodge decomposition on $X_0$ gives $q(X_0) = m.$ 
Let 
\[
\alpha\colon \sX \to \sA
\]
be the relative Albanese map. 
Then $\sA \to \Delta $ is a torus bundle and 
\[
\alpha_t = \alpha \vert X_t\colon  X_t \to A_t
\]
is the Albanese map for all $t.$
Since $\alpha_t$ is surjective for all $t \ne 0,$ the map $\alpha$ is surjective, too, and so is $\alpha_0$.  
We may choose relative vector fields
\[
v_1, \ldots, v_m \in H^0(\sX,T_{\sX/\Delta}),
\]
such that for all $t,$ the push-forwards $(\alpha_t)_*(v_i \vert X_t)$ form a basis of $H^0(A_t,T_{A_t}).$ 
Since singular fibers are mapped to singular fibers by automorphisms of $X_0,$ it follows that the singular locus $S$
of $\alpha_0$, i.e., the set of points $a \in A_0$ such that the fiber over $a$ is singular, is invariant in $A_0$ under
the automorphism group. Hence $S = \emptyset$, so that $\alpha_0$ is a submersion like all the other maps $\alpha_t.$ 
Therefore $\alpha$ is a bundle, with fibers $ \bP_n$ over $\Delta^*.$ The global rigidity of the projective space 
\cite{Si89}
applied to local sections of
$\sA$ over $\Delta,$ passing through $A_0$, implies that all fibers of $\alpha$ are $ \bP_n$. The existence of $\mathcal \sE$ follows from \cite{El82}, (4.3). 
\end{proof} 

\begin{remark} {\rm  Popovici \cite{Po09} has shown that any global deformation of projective manifolds is in class $\sC$, so that the  
assumption in Theorem 6.1 that $X_0$ is in class $\sC$ can be removed in case $X_t$ is projective. The K\"ahler version of Popovici's theorem is still open.} 
\end{remark}  

\begin{example} {\rm  We cannot conclude in Theorem 6.1 that $X_0 \simeq A_0 \times \bP_n, $ even if $m = n = 1$. Take e.g.~a rank-2 vector bundle
$\sF$ over $\bP_1  \times \Delta$ such that $\sF \vert \bP_1 \times \{t\} = \sO^2 $ for $t \ne 0$ and $\sF \vert \bP_1  \times \{0\} = \sO(1) \oplus \sO(-1)$. 
Let $\eta\colon A \to \bP_1$ be a two-sheeted covering from an elliptic curve $A$ and set $\sE = (\eta \times {\rm id})^*(\sF).$ Then $\sX = \bP(\sE)$ is a family
of compact surfaces $X_t$ such that $X_t = A \times \bP_1$ for $t \ne 0$ but $X_0$ is not a product. Notice also that $X_0$ is not almost homogeneous.\\
It is a trivial matter to modify this example to obtain a map to a $2$-dimensional torus which is a product of elliptic curves. Therefore the limit of a K\"ahler contact manifold
with positive irregularity might not be a contact manifold again. }
\end{example} 

\begin{corollary} Assume the situation of Theorem 6.1. Then the following assertions are equivalent.
\begin{enumerate} 
\item $X_0 \simeq A_0 \times \bP_n$.
\item $\sE_0$ is semi-stable for some K\"ahler class $\omega$.
\item $X_0$ is homogeneous.
\item $X_0$ is almost homogeneous. 
\end{enumerate}
\end{corollary}

\begin{proof} (1) implies (2). Under the assumption of (1), there is a line bundle $L$ on $A_0$ such that $\sE_0 \simeq L^{\oplus n+1}.$ Hence $\sE$ is semi-stable 
for actually any choice of $\omega.$  \\
(2) implies (3). From the semi-stability of $\sE_0$ and $h^0(\sE_0) \geq n+1,$ it follows easily that $\sE_0$ is trivial and that $X_0 $ is homogeneous as product
$A_0 \times \bP_n$. In fact, choose $n+1$ sections of $\sE_0$ and consider the induced map $\mu\colon \sO_{A_0}^{n+1} \to \sE_0$. 
By the stability of $\sE_0$, the map $\mu$ is generically surjective. Hence $\det \mu \ne 0$, hence an isomorphism, so that
$\mu$ itself is an isomorphism. \\ 
The implication ``(3) implies (4)'' is obvious. \\
(4) implies (1). Consider the tangent bundle sequence
\[
0 \to T_{X_0/A_0} \to T_{X_0} \to \alpha_0^*(T_{A_0}) \to 0.
\]
Since $X_0$ is almost homogeneous, all vector fields on $A_0$ must lift to $X_0$.
Consequently the connecting map
\[
H^0(X_0,\pi^*(T_{A_0})) \to H^1(X_0,T_{X_0/A_0})
\]
vanishes, and therefore the tangent bundle sequence splits.
Let $\sF = \alpha_0^*(T_{A_0}). $ Then $\sF \subset T_{X_0}$ is clearly a foliation and it has compact
leaves (the limits of tori in $A_t  \times \bP_n$). By \cite{Hoe07}, 2.4.3, there exists an equi-dimensional holomorphic map $f\colon X_0 \to Z_0 $
to a compact variety $Z_0$  such that the set-theoretical fibers $F$ of $f$ are leaves of $\sF$. Since the fibers $F$ have an \'etale map to $A_0$, they must be  tori again. 
It is now immediate that $Z_0 = \bP_n$ and that $X_0 = A_0 \times \bP_n$. 
\end{proof}

\begin{corollary} \label{corsplit} Assume in Theorem 6.1 that $m = 2$ and $n = 1$. Then either $X_0 \simeq A_0 \times \bP_1$, or $X_0 = \bP(\sE_0)$ and one of the
following holds: 
\begin{enumerate} 
\item There is a torus bundle $p\colon A_0 \to B_0$ to an elliptic curve $B_0$ and the rank-2 bundle $\sE_0$ on
$A_0$ sits in an extension
\[
0 \to p^*(\sL_0) \to \sE_0 \to p^*(\sL_0^*) \to 0
\]
with $\deg \sL_0 > 0$.
\item The rank 2-bundle $\sE_0$ sits in an extension 
$$ 0 \to \sS \to \sE_0 \to \sI_Z \otimes \sS^* \to 0 $$
with an ample line bundle $\sS$ and a finite non-empty set $Z$ of degree $\deg Z = c_1(\sS)^2.$ 
\end{enumerate} 
\end{corollary} 

\begin{proof} By Corollary 6.4 we may assume that $\sE_0$ is not semi-stable for some (or any) K\"ahler class $\omega.$ 
Let $\sS$ be a maximal destabilising subsheaf, which is actually a line bundle, leading to an exact sequence 
\[
0 \to \sS \to \sE_0 \to \sI_Z \otimes \sS^* \to 0.
\]
Here $Z$ is a finite set or empty. Taking $c_2$ and observing that $c_2(\sE_0) = 0$ gives
\[
c_1(\sS)^2 = \deg Z.
\]
The destabilisation property reads
\[
c_1(\sS) \cdot \omega > 0.
\]
Since $h^0(\sE_0) \geq 2$, we deduce that $h^0(\sS) \geq 2, $ in particular, $\sS$ is nef, $\sS$ being maximal destabilizing. \\
If $\sS$ is ample, there is nothing more to prove, hence we may assume that $\sS$ is not ample. $\sS$ being nef, $c_1(\sS)^2 = 0$
and $\sS$ defines a submersion
$p\colon A_0 \to B_0$ to an elliptic curve $B_0$ such that there exists an ample line bundle $\sL_0$ with $\sS = p^*(\sL_0)$. 
Therefore we obtain an extension
\[
0 \to p^*(\sL_0) \to \sE_0 \to p^*(\sL_0^*) \to 0,
\]
as required.
\end{proof} 

\begin{remark} {\rm The second case in Corollary \ref{corsplit} really occurs. Take a finite map $f\colon \sA \to \bP_2 \times \Delta$ 
over $\Delta$ and a rank-2 bundle $\sF $ on $\bP_2 \times \Delta$ such that $\sF \vert \bP_2 \times \{t\} \simeq \sO^2$ 
for $t \ne 0$ and such that $\sF_0$ is not trivial. For examples see e.g.~\cite{Sc83}. Now $\sE = f^*(\sF)$ gives an example we are looking for. }
\end{remark}

\begin{corollary} Assume in Theorem 6.1 that $m = 2$ and $n = 1$. 
Let $\Phi\colon T_{\sX/\Delta} \to  {{-K_{\sX}} \over {2}} $ be a morphism such that $\Phi \vert X_t = \phi_t$ is a contact morphism 
(i.e., defines a contact structure) for $t \ne 0.$ Suppose that 
\[
\phi_0\colon  T_{X_0} \to {{-K_{X_0}} \over {2}}
\]
does not vanish identically. 
Then the kernel $\sF_0$ of $\phi_0$ 
is integrable (in contrast to the maximally non-integrable bundle $\sF_t$).
\end{corollary} 

\begin{proof} 

 We consider a family $(\phi_t)$ of morphisms
\[
\phi_t\colon T_{X_t} \to \sH_t
\]
such that $\phi_t$ is a contact form for $t  \ne 0$ and $-K_{X_t} = 2 \sH_t.$ 
Consider the (torsion free) kernel $\sF_0$ of $\phi_0.$  We need to show that the induced map 
\[
\mu\colon (\bigwedge^2 \sF_0)^{**} = \det \sF_0 \to \sH_0.
\]
vanishes. 
Since the determinant of the kernel $\sF_t$ of $\phi_t$ is isomorphic to $\sH_t,$ we conclude that 
\begin{equation}\label{detiso}
\det  \sF_0 \simeq  \sH_0 \otimes \sO_{X_0}(E) \tag{$*$}
\end{equation}
with an effective (possibly vanishing) divisor $E$ on $X_0.$ 
Now the induced map 
\[
\mu\colon \det \sF_0 \to \sH_0
\]
must have zeroes, otherwise $X_0$ would be a contact manifold, hence $X_0  \simeq A_0 \times \bP_1.$ 
Thus $\mu = 0$ by~\eqref{detiso}, and $\sF_0$ is integrable. 
\end{proof} 

\vfill \eject

%
%%%%%%%%%%%%%%%%%%
%%% REFERENCES %%% 
%%%%%%%%%%%%%%%%%%
%
\providecommand{\bysame}{\leavevmode\hbox to3em{\hrulefill}\thinspace}

\end{document}